\documentclass[a4paper,11pt]{amsart}

\usepackage[utf8]{inputenc}

\usepackage[english]{babel}
\usepackage{comment}

\usepackage{amsmath, amssymb, amsthm, thmtools}
\usepackage[T1]{fontenc}

\usepackage{graphicx}

\usepackage{hyperref}
\usepackage[capitalize]{cleveref}

\usepackage{tikz}
\usetikzlibrary{matrix}
\usetikzlibrary{positioning}
\usetikzlibrary{cd}
\usetikzlibrary{arrows}

\newtheorem{lemma}{Lemma}[section]
\newtheorem{proposition}[lemma]{Proposition}
\newtheorem{theorem}[lemma]{Theorem}

\theoremstyle{definition}
\newtheorem{definition}[lemma]{Definition}

\newtheorem{example}[lemma]{Example}

\newtheorem{remark}[lemma]{Remark}

\makeatletter
\newcommand{\dashedrightarrow}[1][2pt]{%
  \settowidth{\@tempdima}{$\longrightarrow$}\rightarrow
  \makebox[-\@tempdima]{\color{white}\rule[0.5ex]{#1}{1pt}\,}
  \phantom{\longrightarrow}
}
\makeatother

\usepackage{ifthen}
\newcommand{\stacks}[2][]{\cite[\href{http://stacks.math.columbia.edu/tag/#2}{\ifthenelse{\equal{#1}{}}{Tag #2}{#1, Tag #2}}]{stacks-project}}  

\let\phi\varphi
\let\varphi\phi

\newcommand{\QQ}{\mathbb{Q}}
\newcommand{\GG}{\mathbb{G}}
\newcommand{\PP}{\mathbb{P}}

\newcommand{\spec}{\operatorname{Spec}}

\newcommand{\End}{\operatorname{End}}

\newcommand{\im}{\text{im}}

\newcommand{\id}{\mathrm{id}}
\newcommand{\Hom}{\operatorname{Hom}}

\newcommand{\Gal}{\operatorname{Gal}}

\newcommand{\Res}{\operatorname{Res}}
\newcommand{\genus}{\operatorname{genus}}

\newcommand{\bijection}{\overset{\sim}{\longrightarrow}}
\newcommand{\isomorphism}{\bijection}
\newcommand{\rationalmap}{\dashedrightarrow}

\renewcommand{\tilde}{\widetilde}

\newcommand{\bmat}{\begin{pmatrix}}
\newcommand{\emat}{\end{pmatrix}}

\newcommand{\const}{^{\operatorname{const}}}
\newcommand{\alg}{^{\operatorname{alg}}}

\newcommand{\Jac}{\operatorname{Jac}}
\renewcommand{\div}{\operatorname{div}}

\newcommand{\dd}{\mathrm{d}}

\newcommand{\Lau}[1]{(\!(#1)\!)}

\let\oldsubset\subset
\let\subset\subseteq

          \newtheorem{remarks}[definition]{Remarks}

\title{Autonomous first order differential equations}

\author[M.\,P.\,Noordman]{Marc Paul Noordman}
\address{Bernoulli Institute, University of Groningen, P.O. Box 407, 9700 AG Groningen, The Netherlands.}
\email{m.p.noordman@rug.nl, m.van.der.put@rug.nl, j.top@rug.nl}

\author[M.\,van der Put]{Marius van der Put}

\author[J.\,Top]{Jaap Top}

\date{April 17, 2019}



\begin{document}

\maketitle

\begin{abstract}  The problem of algebraic dependence of solutions to (non-linear) first order autonomous equations over an algebraically closed field of characteristic zero is given a `complete' answer, obtained independently of model theoretic results on differentially closed fields. Instead, the geometry of curves and generalized Jacobians
provides the key ingredient.  Classification and formal solutions of autonomous equations are treated. The results are applied to answer a question on $D^n$-finiteness of solutions of first order differential equations.

\end{abstract}

\section{Introduction and summary}\label{1}

Let $C$ be an algebraically closed field of characteristic zero. For a linear differential equation of order $n$ over (say) the differential field 
$(C(z),\frac{\dd}{\dd z})$, there is an extension of differential fields $PVF\supset C(z)$, called a Picard--Vessiot field,
 such that $C$ is also the field of constants of $PVF$ and $PVF$ is `minimal' such that the equation has  linearly independent solutions
 $y_1,\dots ,y_n$ over $C$. For any extension of differential fields $PVF\subset L$, the  $y_1,\dots ,y_n$ form  a basis, now over the field of constants of $L$, of all solutions in $L$ of the equation.  The transcendence degree of $PVF$ over $C(z)$
 is at most $n^2$. 
 
   The algebraic relations between these solutions and their derivatives are measured by the differential Galois group. This linear algebraic group 
 subgroup of ${\rm GL}_n(C)$ will, in general, become smaller if $C(z)$ is replaced by a differential field extension. For details, see for example \cite{vdP-S}. \\

   For  non-linear differential equations, the theory of Galois groupoids,  developed by B.~Malgrange, H.~Umemura and many others, replaces the usual differential Galois groups (see for instance \cite{Casale} and references therein). It is concerned with
 solutions and their algebraic relations. Here we will study the simplest type of non-linear equations: autonomous first order differential equations
  \[ P(u, u') = 0 \]
 where $P\in C[X,Y]$ is irreducible, involves both $X$ and $Y$ and $u'$ stands for $\frac{\dd u}{\dd z}$.  
The theory of Galois groupoids does not seem to shed much light on these equations. The following example of M.~Rosenlicht 
\cite{R1} shows that non-linear differential equations are rather different from
linear ones. It states that any set of non-constant solutions in any differential field (say, a differential extension field of $C(z)$) of the differential equation
\[ u' = u^3 - u^2\]
is {\it algebraically independent over} $C$. We give a self-contained elementary proof of this in \cref{Lemma Rosenlicht} and \cref{Rosenlichts example}.

Rosenlicht's result was found in relation with the model theory of {\it 
differentially closed fields in characteristic zero} (DCF$_0$). It has as consequence that the differential closure
$\mathcal{U}$ of, say, $\mathbb{Q}$, admits proper differentially closed subfields. Since then, first order autonomous
differential 
equations have become an important subject in the study of the model theory DCF$_0$. The papers 
\cite{F-D-R,F-M,H-I,N-P}  of E. Hrushovsky, M. Itai, J. Freitag and others are concerned with algebraic relations between solutions of autonomous equations. The language and the proofs are embedded in model theory. \\

The present paper aims to prove results on algebraic relations between  solutions, independently of the model theoretic framework. Instead the main tools are the geometry of curves and 
their (generalized) Jacobian varieties.  One observes that $P\in C[X,Y]$ and $P(u,u')=0$ defines a differential field, the fraction field $C(x,y)$ of  $C[x,y]=C[X,Y]/(P)$  with the derivation $D$ given by $D(x)=y$. Let the pair $(X,\omega)$ denote the 
(irreducible, smooth, projective) curve $X$ over $C$ with function field $C(x,y)$ and $\omega$ the rational 1-form on $X$,
dual to the rational derivation $D$ on $X$.  The relation between $D$ and $\omega$ can be expressed as 
$\omega=\frac{\dd f}{D(f)}$ for any $f\in C(x,y)\setminus C$.

Two first order autonomous differential equations are considered to be 
``the same''  if the associated
pairs $(X_i,\omega_i)$ for $i=1,2$ are isomorphic, in the sense that there exists an isomorphism $\phi\colon X_1 \isomorphism X_2$ such that $\omega_1 = \phi^*\omega_2$. Moreover, by Lemma 5.2 in \cite{N-N-P-T} any $(X,\omega)$ (with $\omega \neq 0$) is associated to some autonomous equation.\\

\noindent {\it Example}: $u'=u^3-u^2$ corresponds to $(\mathbb{P}^1,\frac{\dd x}{x^3-x^2})$ and $D=(x^3-x^2)\frac{\dd }{\dd x}$. For $C=\mathbb{C}$, an analytic solution
$u(z)$ satisfies a formula  $z+c=\int ^{u(z)}_*\frac{\dd x}{x^3-x^2}$. Thus in the complex analytic case, solutions to a first order autonomous differential equation
are inverse functions of integrals of algebraic differential forms. This classical topic includes the Weierstrass functions.  \\

Let $D$ be the polar divisor of a given rational 1-form $\omega$ on $X$. Then $\omega$ can be identified with a translation invariant regular 1-form on the \emph{generalized Jacobian variety} $\Jac(X,D)$ of $X$ with respect to the effective divisor $D$. See \cite{Serre-AlgebraicGroups} for the theory of generalized Jacobian varieties.  
Therefore generalized Jacobian varieties enter our study in a natural way.\\

 In the next section, we state our main results (Theorem 2.1 and 2.1b) and present the proof, modulo technical aspects that are worked out in Sections 3 and 4. In Section 5 we discuss existence of ``new forms''. Results on formal solutions and differentially closed fields are proven and discussed in Section 6. The final section applies the theory to answer a question about $D^n$-finiteness of solutions of differential equations.

 \section{Main results and arguments}

  We may identify any non-constant solution $u$ of $P(u,u') = 0$ in a differential field $(F,\  ')$ containing $C$ 
 with the $C$-linear homomorphism of differential fields $\phi\colon C(x,y)\rightarrow F$ given by $\phi (x)=u$ and
 $\phi(y)=u'$. Therefore, given a pair $(X, \omega)$ we define a \emph{non-constant solution in $F$} to be a $C$-linear morphism of differential fields $C(X) \to F$. Given a collection $(X_i, \omega_i)$ of differential equations, we call a set $\{\phi_i\}$ of solutions $C(X_i) \to F$ \emph{algebraically independent over $C$} if for some choice of non-constant $x_i \in C(X_i)$ the set $\{\phi_i(x_i)\} \subset F$ is algebraically independent over $C$. Note that this notion is independent of the choice of the coordinate $x_i \in C(X_i)$, because each $C(X_i)$ has transcendence degree 1 over $C$. Besides non-constant solutions, in Section~6 we will also define and study constant solutions. 
 
 As in the linear case, we are mostly interested in differential extensions fields $F\supset C(z)$ such that $C$ is the field of constants of $F$, e.g., the algebraic closure $C\Lau{z}\alg$ of $C\Lau{z}$ (see \cref{Section-Formal-Solutions}). In model theory one  chooses
 for $F$ a differential closure of, say, $C(z)$.\\
 
  Consider for $i=1,2$ the pairs $(X_i,\omega_i)$. We call
  $(X_2,\omega _2)$ a \emph{pull back} of $(X_1,\omega _1)$ if there is a non-constant morphism
  $m\colon X_2\rightarrow X_1$ such that $m^*\omega_1=\omega_2$.  In terms of function fields, this means that the homomorphism $\phi \colon C(X_1)\rightarrow C(X_2)$
  induced by $m$ is a morphism of differential fields for the derivations dual to the $\omega_i$.  
  
  For an equation $(X,\omega)$ and a differential field $F\supset C$ we write $(X,\omega)(F)^{nc}$ for the set of non-constant solutions of $(X,\omega)$ in $F$.   If $(X_2,\omega _2)$ is a pull back of $(X_1,\omega _1)$,
then $\phi$ induces a map  $(X_2,\omega_2)(F)^{nc}\rightarrow (X_1,\omega_1)(F)^{nc}$. If $F=F^{alg}$, then
all fibres consist of $[C(X_2):C(X_1)]$ elements. Indeed, the ramification points of $\phi$ are all defined over $C$, and hence do not correspond to non-constant solutions.\\

\noindent We consider the following types of autonomous differential equations $(X, \omega)$: \label{listoftypes}
(1). $(X, \omega)$ is \emph{exact} if $\omega = \dd f$ for some $f \in C(X)$;\\
    equivalently, $\omega$ is a pull-back of a translation-invariant 1-form on $\GG_a$.\\
(2). $(X, \omega)$ is \emph{of exponential type} if $\omega = \frac{\dd f}{cf}$ for some $f \in C(X)$ and $c \in C^*$;\\
    equivalently, $\omega$ is a pull-back of a translation-invariant 1-form on $\GG_m$. \\
(3). $(X, \omega)$ is \emph{of Weierstrass type} if $\omega = \frac{\dd f}{g}$ for $f, g \in C(X)$,  $g^2 = f^3 +af + b$ for $a, b \in C$ with $4a^3 + 27b^2 \neq 0$; equivalently it is a pull-back of a translation-invariant 1-form on an elliptic curve. \\
(4).  In the remaining case, $(X,\omega)$ is called \emph{of general type}.\\
 
 A similar classification is present in \cite{F-D-R, H-I}. For example, one easily verifies that Rosenlicht's equation 
 $(\mathbb{P}^1,\frac{\dd x}{x^3-x^2})$ is of general type. We note that the autonomous equations satisfying the Painlev\'e property  are (after normalisation, \cite{M-P}): 
 $u'=1$, $u'=cu$ for $c\in C^*$, and $(u')^2= u^3+au+b$ for $a,b\in C$ with $4a^3+27b^2\neq 0$. This means that (1)--(3) are the pull backs of these basic equations. 
 
 The solutions for equations of type (1)-(3) in a sufficiently large differential field $F\supset C$ are easily found. In case (1), one identifies $C(f)$
with the subfield $C(z)\subset F$ with $z'=1$; a non-constant solution $\phi$ maps $f$ to $z+c$ for some $c\in C$ and extends in finitely many ways to a  differential homomorphism $\phi \colon C(X)\rightarrow C(z)^{alg}\subset F$. 
The cases (2) and (3) are similar, with $C(z)$ replaced by respectively $C(e^{cz})$ and $C(E)$ where $E$ is an elliptic curve.\\  
 
 An interesting algorithmic question is how to decide, for a given pair $(X, \omega)$, its place in the above classification. For the exact case, this is an easy application of Coates algorithm (\cite{Co}). The exponential case can be handled as in the proof of Lemma 6.6 in \cite{B-D}. However, we are not aware of an algorithmic way to distinghuish between the Weierstrass and the general case. \\
 
As in \cite{F-D-R}, we call $(X,\omega)$ {\it old} if it is  a proper pull back (degree $\geq 2$); otherwise $(X,\omega)$
is called {\it new}.  \cref{proposition 2.2} below shows that every  equation of general type is, in a unique way, the pull back of a new equation of general type. 
{\it The main result} of this paper, which generalizes results from
\cite{R1}, \cite{H-I}, is: 

\begin{theorem} \label{theorem 1.1}
 Let $(X,\omega)$ be of general type and new. Any set  $\{\rho_1,\dots ,\rho_n\}$ of distinct non-constant solutions 
 of $(X, \omega)$ in a differential field  $F\supset C$ is algebraically independent over $C$.\end{theorem}

\begin{remark} \label{ramified covering}
Let $(X,\omega)$ be any pair of general type. By \cref{proposition 2.2} below, it is the pull back of a $(X',\omega ')$ which is new and of general type.
 Let $F\supset C$ be an algebraically closed differential field. The pull back function induces a ``ramified covering''
$\tau\colon  (X,\omega)(F)^{nc}\rightarrow (X',\omega ')(F)^{nc}$ of degree $[C(X):C(X')]$ and the elements
of $(X',\omega')(F)^{nc}$ are algebraically independent over $C$. Thus elements of
$(X,\omega)(F)^{nc}$ lying in the same fiber of $\tau$ are algebraically dependent over $C$ (with a relation of degree
bounded by $[C(X):C(X')]$) and elements in distinct fibers are algebraically independent over $C$. One can
compare this with the various properties defined in \cite{F-D-R,F-M,H-I,N-P,Mc,S}).
\end{remark}
 
 \begin{remark}\label{finitely many solutions}
 Let $(X,\omega)$ be of general type. Then for any differential field $F \supset C$ of finite transcendence degree over $C$, the set $(X,\omega)(F)^{nc}$ of non-constant solutions is finite. If $(X,\omega)$ is new, this follows immediately from Theorem 1.1. In the general case, one applies \cref{proposition 2.2} to show that $(X, \omega)$ is the pull-back of a new and general type $(X', \omega')$ via a map $\tau: X \to X'$. The reasoning in \cref{ramified covering} now shows that 
 $\#(X,\omega)(F)^{nc} \leq \deg(\tau) \cdot \#(X',\omega')(F)^{nc}$ is finite.
 \end{remark}

 In fact, a stronger statement than \cref{theorem 1.1} will be proved, namely the following: \\
 
 \noindent {\bf Theorem 2.1b.} {\it Let $(X_1, \omega_1), \ldots, (X_n, \omega_n)$ all be of general type and new. Let $\rho_i$ be a non-constant solution of $(X_i, \omega_i)$ in some fixed differential field $F \supset C$, for $i=1,\ldots, n$. If the set $\{\rho_1, \ldots, \rho_n\}$ is algebraically dependent over $C$, then there are indices $i \neq j$ and an isomorphism $\phi\colon X_i \isomorphism X_j$ such that $\phi^* \omega_j = \omega_i$ and $\rho_i \circ \phi^* = \rho_j$. }\\
 
 \begin{remark}\label{strengthening}
  The statement of Theorem 2.1b can be strengthened: to the list $(X_1, \omega_1), \ldots, (X_n, \omega_n)$ one can add one copy of the pair $(\PP^1, \dd t)$, several pairs of the form $(\PP^1, \frac{\dd t}{c_it})$ for elements $c_i \in C^*$ that are linearly independent over $\QQ$, several pairs of the form $(E_i, \omega_i)$ with the $E_i$ pairwise non-isogenous elliptic curves, etc.  \hfill $\square$
 \end{remark}

 We show that Theorem 2.1b implies Theorem~2.1. Indeed, let $\{\rho_1,\ldots, \rho_n\}$ be non-constant solutions of $(X,\omega)$ in a differential field $F \supset C$, where $(X,\omega)$ is of general type and new. Assume that the set of $\rho_i$ is algebraically dependent over $C$. Applying Theorem 2.1b, we obtain indices $i$ and $j$ such that $\rho_i$ and $\rho_j$ differ only by an automorphism $\phi\colon X \isomorphism X$. Moreover, this automorphism satisfies $\phi^*\omega = \omega$. We claim that this implies that $\phi$ has finite order. Indeed, the only cases of an automorphism $\sigma$ of infinite order are:\\
(i) $\mathbb{P}^1,\ \sigma (x)=x+1$; (ii) $\mathbb{P}^1, \ \sigma (x)=qx$, $q\in C^*$ not a root of unity; \\ 
(iii) $X$ is an elliptic curve and $\sigma$ is a translation over a non-torsion point.\\
Any $\omega$ invariant under such an automorphism can be written as $f\dd x$, $f\frac{\dd x}{x}$ or $f\frac{\dd x}{y}$ for the cases respectively, for some $f$ in the 
corresponding function field. If $\omega$ is $\sigma$-invariant, then so is $f$ and thus $f$ is constant. Therefore, $\sigma$-invariant $\omega$'s in these cases are the (1)--(3) in the classification given on page~\pageref{listoftypes}. Since $(X,\omega)$ is of general type, $\phi$ has finite order. Therefore $\omega$ descends to the quotient curve $X/\langle \phi\rangle$. This contradicts the assumption that $\omega$ is new unless $\phi$ has order 1. Therefore $\phi = \id_X$, and we get that $\rho_i = \rho_j$. \\
 
 The proof of Theorem 2.1b proceeds in two steps. First, we establish the theorem for the case $n = 2$. Then we reduce to this case, by inductively showing that an algebraic relation between $n$ solutions implies an algebraic relation between two of those solutions. \\

\noindent {\it Translating algebraic dependence of solutions into geometry}.\\
 For $i=1,2$, let $\rho_i\colon C(X_i)\rightarrow F$ be a non-constant solution of the equation $(X_i,\omega_i)$. Suppose that these are algebraically dependent over $C$.
This means that the subfield $F_0$ of $F$ generated by $\im(\rho _1)\cup \im(\rho_2)$ is finite over 
$\im(  \rho_1)$ and over $\im( \rho_2)$. Now $F_0$ is invariant under the derivation of $F$ since this holds for
both $\im(\rho_1)$ and $\im(\rho_2)$. Then $F_0$ with its derivation and the maps $\rho_i$ defines a pair 
$(Y,\omega)$ and morphisms $\phi_i\colon (Y,\omega)\rightarrow (X_i,\omega_i)$ such that 
$\omega =\phi_1^*\omega_1=\phi_2^*\omega_2$. The key step in the proof of \cref{theorem 1.1} is the following theorem, the proof of which is deferred to Section~3:\\

  \noindent 
 {\bf Theorem \ref{theorem 2.1}}. {\it There is a curve $\tilde X$ over $C$, a meromorphic 1-form $\tilde \omega$ on $\tilde X$ and finite morphisms $\psi_i\colon X_i \to \tilde X$ such that the diagram}
\begin{center}
    \begin{tikzpicture}
      \matrix (m) [matrix of math nodes,row sep=2em,column sep=2em,minimum width=3em]
      {
         Y & X_1 \\
         X_2 &  \tilde{X}\phantom{{}_{\frac{3}{2}}}\\};
      \path[-stealth]
        (m-1-1) edge node [above] {$\phi_1$} (m-1-2)
                edge node [left]  {$\phi_2$} (m-2-1)
        (m-1-2) edge node [right] {$\psi_1$} (m-2-2)
        (m-2-1) edge node [below] {$\psi_2$} (m-2-2);
    \end{tikzpicture} 
\end{center}
{\it is commutative and such that $\omega_i = \psi_i^*\tilde \omega$ for $i = 1,2$. }  \\

If in the above theorem the pairs $(X_1, \omega_1)$ and $(X_2, \omega_2)$ are both new, then both $\psi_1$ and $\psi_2$ are isomorphisms. This establishes Theorem 2.1b for $n = 2$.

\begin{proof}[Proof of the induction step for Theorem 2.1b]$ $\\
We now use induction to prove Theorem 2.1b for any $n > 2$. We may suppose  that the differential field $F$ is algebraically closed. 
Non-constant solutions $\rho_1,\dots ,\rho_n$ of $(X_1, \omega_1), \ldots, (X_n, \omega_n)$ in $F$ with $n>2$ are given. By induction, we may assume that all proper subsets of $\{\rho_1, \ldots, \rho_n\}$ are algebraically independent. 
Let $K$ be the smallest algebraically closed, differential subfield of $F$ containing the 
solutions $\rho_3,\dots ,\rho_n$. Write $C(X_i)=C(x_i,y_i)$ for $i = 1,\ldots, n$. Then $K$ is the algebraic closure of 
$C(\rho_3(x_3),\rho_3(y_3),\dots ,\rho_n(x_n),\rho_n(y_n))$. The induction hypothesis implies that the transcendence degrees 
of $K(\rho_1(x_1),\rho_1(y_1))$ and $K(\rho_2(x_2),\rho_2(y_2))$ over $C$ are both $n-1$.

 Let $\tilde{C}\supset C$
be the constants of the algebraic closure of $K(\rho_2(x_2),\rho_2(y_2))$. If $\tilde{C}\neq C$, then  
the solutions $\rho_2,\dots ,\rho_n$ of $(\tilde{C}\times _CX_2,\omega_{2,\tilde{C}}), \ldots, (\tilde{C}\times _CX_n,\omega_{n,\tilde{C}})$ are algebraically dependent over $\tilde C$. This contradicts the induction hypothesis since by \cref{3.2} parts (4) and (5), $(\tilde{C}\times _CX_i,\omega_{i,\tilde{C}})$ is also
of general type and new for each $i$. Thus $\tilde{C}=C$.\\

Assume that $\rho_1,\dots ,\rho_n$ are algebraically dependent over $C$.\\
Then $\rho_1,\rho_2$ are algebraically dependent over $K$ and the differential field 
$F_0:=K(\rho_1(x_1),\rho_1(y_1),\rho_2(x_2),\rho_2(y_2))$ is a finite extension of the two fields $K(\rho_1(x_1),\rho_1(y_1))$ and
$K(\rho_2(x_2),\rho_2(y_2))$.

Consider for each $i$ the pair $(X_{i,K},\omega_{i,K})$, where $X_{i,K}=K\times_CX_i$ and $\omega_{i,K}$ is $\omega_i$ regarded as a 1-form on $X_{i,K}$ over $K$. Let $Y$  denote the curve over $K$ with function field $F_0$.
Now $\rho_1$ and $\rho_2$ define morphisms $m_i\colon Y\rightarrow X_{i,K}$. Then $m_i^*\omega_{i,K}, \ i=1,2$ are
two elements of the 1-dimensional $F_0$-vector space $\Omega_{F_0/K}$. Hence $m_1^*\omega_{1,K} =c\cdot m_2^*\omega_{2,K}$
for some $c\in F_0^*$.

 The derivation $D$ of $F_0$ as subfield of $F$ satisfies $D(K)\subset K$. According to \cref{3.1},
there is a natural action, denoted by $\tilde{D}$, of $D$ on $\Omega^1_{F_0/K}$. The formula for $\tilde{D}$ is $\tilde{D}(\sum_k a_k\dd b_k)=\sum_k (D(a_k)\dd b_k+a_k\dd D(b_k))$ for any element $\sum_k a_k\dd b_k\in \Omega^1 _{F_0/K}$.
In our case, for $C(X_i) = C(x_i, y_i)$ we have $\omega_i = \frac{\dd x_i}{D(x_i)}$, and so the derivation $D_i$ on $C(X_i)$ is given by $D_i = D(x_i) \frac{\dd}{\dd x_i}$. One easily
computes that $\tilde{D_i}(\omega_i)=0$. Then also $\tilde{D}(m_i^*\omega_{i,K})=0$ for $i=1,2$. 
It follows that $D(c)=0$ and thus $c\in (F_0\const)^* = C^*$.\\

Now the pull backs of $c\cdot \omega_{1,K}$ by $m_1$ and of $\omega_{2,K}$ by $m_2$ coincide.  According to \cref{3.2},
$(X_{i,K},\omega_{i,K})$ is of general type and new for $i=1,2$. Then applying Theorem \ref{theorem 2.1} to $(X_{1,K}, c\omega_{1,K})$ and $(X_{2,K}, \omega_{2,K})$, one obtains an isomorphism $\phi\colon X_{1,K} \isomorphism X_{2,K}$ such that $m_1\circ \phi =m_2$ and $\phi ^*\omega_K=c\cdot \omega_K$. \\

Applying \cref{3.3}, with $S_i$ the support of the divisor of $\omega_i$ shows that
$\phi$ descends to an isomorphism $X_1 \to X_2$. From $m_1\circ \phi=m_2$ it follows that the subfields
$C(\rho_1(x_1),\rho_1(y_1))$ and $C(\rho_2(x_1),\rho_2(y_1))$ of $F$ coincide. Hence $\rho_1$ and $\rho_2$ are algebraically dependent already over $C$, and so the $n = 2$ case of Theorem 2.1b gives us the result. 
\end{proof}


\section{A geometric theorem}
\begin{theorem}\label{theorem 2.1}
Let $X_1$ and $X_2$ be curves over $C$, and $\omega_1$ and $\omega_2$ non-zero meromorphic 1-forms on $X_1$ and $X_2$ respectively. Suppose that there is a curve $Y$ over $C$ and finite morphisms $\phi_i\colon Y \to X_i$ for $i = 1,2$ such that $\phi_1^*\omega_1 = \phi_2^*\omega_2$. Then there is a curve $\tilde X$ over $C$, a meromorphic 1-form $\tilde \omega$ on $\tilde X$ and finite morphisms $\psi_i\colon X_i \to \tilde X$ such that the diagram
\begin{center}
    \begin{tikzpicture}
      \matrix (m) [matrix of math nodes,row sep=2em,column sep=2em,minimum width=3em]
      {
         Y & X_1 \\
         X_2 &  \tilde{X}\phantom{{}_{\frac{3}{2}}}\\};
      \path[-stealth]
        (m-1-1) edge node [above] {$\phi_1$} (m-1-2)
                edge node [left]  {$\phi_2$} (m-2-1)
        (m-1-2) edge node [right] {$\psi_1$} (m-2-2)
        (m-2-1) edge node [below] {$\psi_2$} (m-2-2);
    \end{tikzpicture} 
\end{center}
is commutative and such that $\omega_i = \psi_i^*\tilde \omega$ for $i = 1,2$. 
\end{theorem}
\begin{proof}
As explained in \S\ref{1}, curves with non-zero 1-forms correspond to function fields with a non-zero $C$-linear derivation. In terms of function fields, the theorem says that there is a subfield of $C(Y)$ of transcendence degree 1 over $C$, closed under the derivation, and contained in (the image of) $C(X_1)$ and $C(X_2)$. The obvious choice of this subfield is the intersection $C(X_1) \cap C(X_2)$. This intersection is closed under the derivation, however, it is not clear that this intersection differs from $C$. We will use the theory of generalized Jacobian varieties (see \cite{Serre-AlgebraicGroups}) to show that $C(X_1) \cap C(X_2) \neq C$.

Let $D_1$ and $D_2$ be the divisor of poles of $\omega_1$ and $\omega_2$ respectively. Choose an effective divisor $E$ on $Y$ with $E \geq \phi_i^*(D_i)$ for $i = 1,2$. Then we have a morphism of commutative group varieties
\begin{align*}
     \Phi\colon \Jac(Y, E) &\longrightarrow \Jac(X_1, D_1) \times \Jac(X_2, D_2) \\
 \mbox{divisor class }     [H] &\longmapsto \big(\,[\phi_1(H)], \,[-\phi_2(H)]\,\big)
\end{align*}
(note the minus sign). Define $A = (\Jac(X_1, D_1)\times \Jac(X_2, D_2))/\im(\Phi)$ (this quotient exists as a group variety by \cite[Th\'eor\`eme 4.C.]{Anantharaman}). The invariant 1-forms on $\Jac(Y, E)$ correspond to 1-forms on $Y$ with poles at most $E$, and the invariant 1-forms on $\Jac(X_1, D_1)\times \Jac(X_2, D_2)$ correspond to pairs of 1-forms on $X_1$ and on $X_2$ with poles at most $D_1$ and $D_2$. Under these identifications, pulling back invariant 1-forms via $\Phi$ corresponds to the map
\begin{align*}
    \Phi^*\colon H^0(X_1, \Omega^1(D_1)) \times H^0(X_2, \Omega^1(D_2)) &\longrightarrow H^0(Y, \Omega^1(E)) \\
    (\alpha_1, \alpha_2) &\longmapsto \phi_1^*\alpha_1 - \phi_2^*\alpha_2 
\end{align*} 
(the minus sign here comes from the minus sign in the definition of $\Phi$). \\
The kernel of this map corresponds to invariant 1-forms on $A$. It is given by
\[V := \ker \Phi^* = \{(\alpha_1, \alpha_2) : \phi_1^* \alpha_1 = \phi_2^* \alpha_2\} \]
In particular do we have $(\omega_1, \omega_2) \in V$. Therefore there is an invariant 1-form on $A$ such that $\omega_1$ and $\omega_2$ are obtained from it via pulling back along rational maps $X_i \rationalmap A$. Since $\omega_1$ and $\omega_2$ are non-zero, this already shows that $A \neq 0$, so $\dim A \geq 1$. If $\dim A  = 1$, then we can take $\tilde X$ to be (the completion of) $A$, and $\tilde \omega$ the 1-form on $A$ corresponding to the pair $(\omega_1, \omega_2)$. 

In what follows, we may therefore assume that $\dim A > 1$. Since we know that $\dim_C V = \dim A$, there exists an element $(\alpha_1, \alpha_2) \in V$ that is $C$-linearly independent of $(\omega_1, \omega_2)$. 

There are unique $f_1 \in C(X_1)$ and $f_2 \in C(X_2)$ such that $\alpha_1 = f_1\omega_1$ and $\alpha_2 = f_2\omega_2$. We claim that at least one of the $f_i$ is not in $C$. Suppose otherwise. Then after scaling we can assume $f_1 = 1$, so that $\alpha_1 = \omega_1$. Then 
\[f_2\phi_2^*(\omega_2) = \phi_2^*(\alpha_2) = \phi_1^*(\alpha_1) = \phi_1^*(\omega_1) = \phi_2^*(\omega_2) \]
so we also get $f_2 = 1$ (since $\omega_2 \neq 0$). This contradicts the assumption that $(\alpha_1, \alpha_2)$ is $C$-linearly independent of $(\omega_1, \omega_2)$. Hence at least one of the $f_i$ is not in $C$. 

Finally we use that $\phi_1^*\omega_1 = \phi_2^*\omega_2$ and $\phi_1^*\alpha_1 = \phi_2^*\alpha_2$, which gives us $\phi_1^* f_1  = \phi_2^* f_2$. This is therefore an element of $C(X_1) \cap C(X_2)$ but not in $C$, and the existence of such an element is what we wanted to show.
\end{proof}

\begin{example} (Compare \cite{H-I} Remark 2.17.) Let $(E_i,\omega_i), \ i=1,2$ denote elliptic curves over $C$ with
regular 1-forms. Theorem \ref{theorem 2.1} implies that there is an algebraic relation between non-constant solutions of 
$(E_1,\omega_1)$ and $(E_2,\omega_2)$ if and only if there is an isogeny $\tau\colon  E_1\rightarrow E_2$ such that
$\tau^*\omega_2\in \mathbb{Q}\cdot\omega_1$. 
\end{example}

\begin{proposition}\label{proposition 2.2} Any differential equation $(X,\omega)$ of general type is, in a unique way, the pull back of a new differential equation of general type.
\end{proposition} 
\begin{proof}
By definition, the classes of exact, exponential type and Weierstrass type 1-forms are closed under pull-back. Hence, if $\omega$ is the pull-back of some $\tilde \omega$, then $\tilde \omega$ must be of general type as well. Hence, it suffices to show that $\omega$ is the pull-back of some new 1-form.

To prove the existence of the pull back $(X,\omega)\rightarrow (X',\omega')$ with a new $(X',\omega')$ we have to show that an infinite chain of pull backs is not possible. For this, let $\phi\colon  X \to X'$ be a morphism and $\omega'$ a 1-form on $X'$ such that $\omega = \phi^*\omega'$. Then $\omega'$ has at least 1 zero, since it is of general type. A local computation shows that $\omega$ then has at least $\deg (\phi)$ zeros (counted with multiplicity). Therefore, $\deg (\phi)$ is bounded by the degree of the divisor of zeroes of $\omega$. Hence infinite chains are not possible. 

The uniqueness follows from
Theorem \ref{theorem 2.1}, together with the fact that $(X', \omega')$ has trivial automorphism group (see the reasoning below Theorem~2.1b). 
\end{proof}
 
 For $\omega$ not of general type, infinite chains of pull backs may occur. There are essentially two cases:\\ 
 (i) $(\mathbb{P}^1,\frac{\dd x}{x})$ because $\frac{\dd x}{x}=\frac{1}{n}\frac{\dd x^n}{x^n}$  and\\
 (ii) $(E,\omega)$ with an elliptic curve $E$ and $\omega$ a regular 1-form. Indeed, one has
  $\omega = [n]^*(\frac{1}{n}\omega)$ where $[n]\colon  E \to E$ is the multiplication-by-$n$ map. 

\section{Odds and ends}
\subsection{The derivation \texorpdfstring{$\tilde{D}$}{Dtilde} on \texorpdfstring{$\Omega^1 _{L/K}$}{Omega1\_L/K}}
 The following result is used in Section~2 and also in \cite[p.~530]{R1}. For 
 convenience and lack of reference we present a short proof.
 
\begin{proposition}\label{3.1}
 Let $L\supset K$ be an extension of differential fields of characteristic zero. Write $D$ for the derivation on $L$.
 There exists an additive map $\tilde{D}\colon \Omega^1 _{L/K}\rightarrow \Omega^1_{L/K}$ satisfying 
 $\tilde{D}(a\cdot \dd b)=D(a)\cdot \dd b+a\cdot \dd D(b)$ for all $a,b\in L$.  
 \end{proposition}
 \begin{proof} Let $\{x_i\}_{i\in I}$ be a transcendence basis of $L/K$. Then $\{\dd x_i\}$ is a basis of the $L$-vector space
 $\Omega^1_{L/K}$. Define $\tilde{D}\colon \Omega^1_{L/K}\rightarrow \Omega^1_{L/K}$ by the formula 
 $\tilde{D}(\sum a_i\cdot \dd x_i)=\sum (D(a_i)\cdot \dd x_i+a_i\cdot \dd D(x_i))$. We want to show that for all $a,b\in L$
 the expression $E(a,b)=\tilde{D}(a\cdot \dd b)-D(a)\cdot \dd b-a\cdot \dd D(b)$ is zero.\\
 
 The derivatives $\frac{\partial}{\partial x_i}$ on $L/K$ are defined by the formula $\dd f=\sum _{i}\frac{\partial f}{\partial x_i}\cdot \dd x_i$.
{\it  The first observation is $E(a,b)=aE(1,b)$}. Indeed, write $a\dd b=\sum _i a\frac{\partial b}{\partial x_i}\cdot \dd x_i$. Then 
\[ E(a,b)=\sum_i  \{D(a\frac{\partial b}{\partial x_i})\cdot \dd x_i+a\frac{\partial b}{\partial x_i}\cdot \dd Dx_i\}
-D(a)\sum _i\frac{\partial b}{\partial x_i}\cdot \dd x_i-a\cdot \dd D(b)\]
\[\mbox{is equal to } 
 a\{\sum _iD(\frac{\partial b}{\partial x_i})\cdot \dd x_i +\frac{\partial b}{\partial x_i}\cdot \dd Dx_i)-\dd D(b)\}=a\{\tilde{D}(\dd b)-\dd D(b)\}.\]
 
  The map $\ell \colon L\rightarrow \Omega^1_{L/K}$, given by $\ell (b)=E(1,b)=\tilde{D}(\dd b)-\dd Db$, can be seen to have the properties:\\
  $\ell$ is additive; $\ell (bc)=b\ell(c)+c\ell (b)$; $\ell$ is zero on $K$ and therefore $K$-linear.
 
   Choose any $L$-linear map $m\colon\Omega^1 _{L/K}\rightarrow L$. Then $m\circ \ell \colon L\rightarrow L$ is a $K$-linear derivation
   on $L$ and is zero on the transcendence basis $\{x_i\}$. Thus $m\circ  \ell=0$ for every $m$. Hence $\ell =0$. \end{proof}

We note that the action of $\tilde{D}$ on $\Omega^1_{L/K}$ induces a natural action on 
$\Hom_{L}(\Omega^1_{L/K},L)=\operatorname{Der}(L/K)$, the $K$-vector space of all $K$-linear derivations $E$ on $L$. This action is simply
$E\mapsto [E,D]$.

 \subsection{Base change}
 
\begin{proposition}\label{3.2}
 Let $X$ be a curve over $C$ and $\omega\in \Omega_{C(X)/C}$ a non-zero rational 1-form on $X$.
 Let $C \subset K$ be an arbitrary field extension. Write $X_K:= K\times _CX $ and let $\omega_K$ denote $\omega$ regarded as 1-form on $X_K$. Then
 \begin{enumerate}
    \item $\omega$ is exact $\Leftrightarrow$ $\omega_K$ is exact.
    \item $\omega$ is of exponential type $\Leftrightarrow$ $\omega_K$ is of exponential type.
    \item $\omega$ is of Weierstrass type $\Leftrightarrow$ $\omega_K $ is of Weierstrass type.
    \item $\omega$ is of general type $\Leftrightarrow$ $\omega_K$ is of general type. 
    \item $\omega$ is old $\Leftrightarrow$ $\omega_K$ is old. 
\end{enumerate}
\end{proposition}
\begin{proof}
Item (4) follows directly from (1), (2) and (3). For the cases (1), (2), (3) and (5) the implication $\Rightarrow$ is clear. The proofs of the implication $\Leftarrow$ for (1), (2), (3) and (5) are very similar. We give a proof for item~(5). 

Let $\tilde X$ be a curve over $K$, let $\phi\colon X_K \to \tilde X$ be a morphism of degree at least 2, and let $\tilde \omega$ be a meromorphic 1-form on $\tilde X$ such that $\phi^*\tilde \omega = \omega_K$. Let $R$ be the $C$-subalgebra of $K$ generated by the coefficients of some finite set of equations for $\tilde X_K$, $\phi$ and $\tilde \omega$. Then by construction, $\tilde X$, $\phi$ and $\tilde \omega$ are defined over $R$. After replacing $R$ by $R[1/f]$ for some non-zero $f$ in the Jacobian ideal of $\tilde X$ over $R$, we may even assume that $\tilde X$ is smooth over $R$. Now take a maximal ideal $\mathfrak m$ of $R$. Then we have $C = R/\mathfrak m$ since $R$ is finitely generated over $C$. Let $\tilde X_C$, $\phi_C$ and $\tilde \omega_C$ be the base changes of $\tilde X$, $\phi$ and $\tilde \omega$ via this map $R \to C$. Then $\tilde X_C$ is a smooth curve over $C$, $\tilde \omega_C$ is a 1-form on $X_C$, and $(\phi_C)^*\tilde \omega_C = \omega$. Moreover, the degree of $\phi_C$ is the same as the degree of $\phi$, because of flatness of $\tilde X$ over $R$. In particular, $\deg \phi > 1$, and so $\omega$ is old. 
\end{proof}

Let $X_1$ and $X_2$ be curves over $C$. Consider an isomorphism $\phi\colon  X_{1,F} \to X_{2,F}$ defined over a field extension $F/C$. We need a result on descending such an isomorphism to $C$. Note that this is not automatic, even in the case that $\phi$ is an automorphism of finite order. As an example, take $X_1 = X_2 = E$ an elliptic curve, and fix a point $P_0 \in E(F) \setminus E(C)$. Then the automorphism $E \to E$ given by $Q \mapsto P_0 - Q$ has order 2, and does not descend to an automorphism of $E$ over $C$. Similar examples exist on $\PP^1$.

\begin{proposition}\label{3.3}
Let $X_1$ and $X_2$ be curves over $C$, $S_i \oldsubset X_i(C)$ finite (possibly empty) subsets of $C$-points of $X_i$ for $i = 1,2$, such that 
\[ \# S_i \geq \begin{cases} 3 & \textrm{if } \genus(X_i) = 0, \\
                           1 & \textrm{if } \genus(X_i) = 1, \\
                           0 & \textrm{if } \genus(X_i) \geq 2. \end{cases} \]
Let $F/C$ be a field extension, and let $\phi\colon X_{1,F} \isomorphism X_{2,F}$ be an isomorphism of curves over $F$ such that $\phi(S_1) = S_2$. Then $\phi$ is defined over $C$, i.e. there is an isomorphism $\tilde \phi\colon X_1 \isomorphism X_2$ such that $\phi$ is the base change of $\tilde \phi$ to $F$.
\end{proposition}
\begin{proof}
Let $I = \mathrm{Isom}(X_1, X_2)$ be the isomorphism scheme from $X_1$ to $X_2$; it is a scheme, locally of finite type over $C$, characterized by the property that for any $C$-algebra $R$ we have $I(R) = \mathrm{Isom}_R(X_{1,R}, X_{2,R})$ (see e.g. \cite[4c]{Grothendieck-Hilbert}). The isomorphism $\phi$ of $X$ restricts by assumption to a bijection $S_1 \to S_2$. Call this bijection $\tau$. Let $I_\tau \subset I$ be the closed subscheme of isomorphisms from $X_1$ to $X_2$ restricting to this same bijection $S_1 \to S_2$, i.e. the closed subscheme of $I$ cut out by the image of the equations $\phi_{univ}(s_I) = \tau(s_I)$ for $s \in S_1$, where $\phi_{univ}$ denotes the universal isomorphism from $X_{1,I} := X_1 \times_C I$ to $X_{2,I} = X_2 \times_C I$, and $s_I$ the $I$-point of $X_{1,I}$ corresponding to $s \in X_1(C)$. (In the case where $S_1 = S_2 = \emptyset$ we have $I_\tau = I$.) Note that by construction, $\phi$ corresponds to an $F$-point of $I_\tau$. Hence, it suffices to prove that the natural map $I_\tau(C) \to I_\tau(F)$ is a bijection. 

Our assumption on the size of $S_1$ and $S_2$ implies that $I_\tau(C)$ is finite. Since $I_\tau$ is locally of finite type over $C$, and $C$ is algebraically closed, we conclude that $I_\tau$ is a finite scheme (since a transcendental extension of $C$ is not of finite type over $C$, all closed points of $I_\tau$ come from $C$ points, and so $I_\tau$ has only finitely many closed points, each with residue field $C$), and so we find that the map $I_\tau(C) \to I_\tau(F)$ is a bijection. 
\end{proof}

\section{Existence of new forms}

For a given curve $X$ and effective divisor $D$ one wants to find the condition implying the existence
of a new form $\omega$ with $\div(\omega)\geq -D$.

\subsection{New forms on \texorpdfstring{$\mathbb{P}^1$}{P1}}
  One associates to $\omega$ the ``divisor'' $R$ (with coefficients in $C$) given by $R:=\sum _{p\in \mathbb{P}^1}\Res_p(\omega) [p]$.  Let $A\subset C$ denote  the group generated by the residues of $\omega$. Choose a basis $a_1,\dots ,a_n$ of $A$. Then one can write
 $R=a_1R_1+\dots +a_nR_n$, where each $R_i$ is a divisor with coefficients in $\mathbb{Z}$. We note that the g.c.d. of the coefficients
 of each $R_i$ is 1.  Since the sum of all residues is zero
 and the $a_1,\dots ,a_n$ are independent, each $R_i$ is a divisor of degree zero and thus equal to the divisor of some $u_i\in C(x)^*$.
 
 Now $\omega -\sum a_i\frac{\dd u_i}{u_i}$ has residues zero and is therefore exact, i.e., equal to $\dd v$. Thus we have
 \[\omega =\sum_{i=1}^n a_i\frac{\dd u_i}{u_i}+\dd v\mbox{ with } n\geq 0,\  u_i\in C(x)^*,\ v\in C(x) .\] 
  After the choice of the basis $a_1,\dots ,a_n$, the terms $u_i$ are unique up to a multiplicative constants and $v$ is unique up to
  an additive constant. A base change of $A$ by a matrix in ${\rm GL}(n,\mathbb{Z})$ induces a ``multiplicative'' change of the
  $\{u_i\}$ by the same matrix. We deduce that the subfield $C(u_1,\dots ,u_n,v)$ of $C(x)$ depends only on 
  $\omega$ and $C(x)$. Moreover, if $\omega$ is new, then one immediately sees that $C(u_1,\dots ,u_n,v)=C(x)$.\\
 

 However, if $C(u_1,\dots ,u_n,v)=C(x)$, then $\omega$ need not be new. Indeed, there are many examples of the type:\\
  \[ \frac{\dd x}{x}+ \sqrt{2}\cdot \frac{\dd (x^6-1)}{x^6-1}+\dd (x^3+x^{-6}) =\frac{1}{3}\frac{\dd x^3}{x^3}+\sqrt{2}\cdot \frac{\dd (x^6-1)}{x^6-1}+\dd (x^3+x^{-6}).\] 
 It seems difficult to give an explicit criterion for deciding whether a given $\omega$ is new. Consider, for example,  
 $u_1,u_2\in C(x)\setminus C$ and $\omega =\frac{\dd u_1}{u_1}+\sqrt{2}\cdot \frac{\dd u_2}{u_2}$. This $\omega$ is not
 new if there are integers $n_1,n_2\geq 1$ such that $C(u_1^{n_1},u_2^{n_2})$ is a proper subfield of $C(x)$.  
 
\subsection{New forms on an elliptic curve \texorpdfstring{$E$}{E}}
 Let $E$ have the affine equation $y^2=x^3+ax+b$ and let $\infty$ denote the point at infinity. For a positive divisor
 $D$ on $E$ of degree $\geq 2$ we give examples for the existence of new forms in the vector space
 $V(D):=\{\omega \ | \ \div(\omega)\geq -D\}$. \\  
 (i) $D=2[\infty ]$.  If $\omega \in V(D)$ is not regular, then $\omega$ is new.\\
 (ii) $D=[P]+[\infty]$. Suppose $\exists f\in C(E)$ with $\div(f)= m\cdot ([P]- [\infty])$ and
$m>1$ (i.e. $P$ has finite order in the group $E(C)$). Then $\omega \in V(D)$ is new if $\omega$ is not regular and not a multiple of $\frac{\dd f}{f}$.\\
 (iii) $D=[P]+[\infty]$. Suppose that the divisor class of $[P]-[\infty]$ has infinite order. Then $\omega\in V(D)$ is new
 if $\omega$ is not regular.\\
 (iv) $V(m[\infty])$ is the direct sum of $\{ \dd f \ |\ f\in C[x,y], \deg _\infty f\geq -m+1\}$ and $V(2[\infty])$.
 Now $\omega \in V(m[\infty ])$ is new if $\omega$ is not regular and not exact.\\ 
 (v) The case of $D=[P_1]+\cdots +[P_d]$ ($d \geq 3$ distinct points) is more complicated. For example, let $E$ be given by the equation $y^2 = (x - a_1)(x-a_2)(x - a_3)$, and take 
 \[ \omega = \frac{\dd x}{x - a_1} + \sqrt{2}\frac{\dd x}{x - a_2}. \]
 Then $\omega$ has three simple poles, namely at the points $(a_1, 0)$, $(a_2, 0)$ and $\infty$, and the residues of these poles are as independent over $\QQ$ as they can be (they sum to zero, but there are no other relations), but nevertheless $\omega$ comes from $\PP^1$.



\subsection{Existence of new forms on higher genus curves}

Let $X$ be a curve of genus $\geq 2$, and let $D \geq 0$ be a divisor on $X$. We write 
\[V(D)=\{\omega \ | \ \div(\omega)\geq -D\},\]
and let $n = \dim V(D) - 1$. The goal of this subsection is to show that a ``generic'' element of $V(D)$ is a new form. More precisely, we have the following result. 

\begin{theorem}
Let $S \subset \PP(V(D)) = \PP^n(C)$ be the image of the collection of old forms in $V(D)$. Then $S$ is contained in a union of finitely many proper projective subspaces and a set of the form $\PP^n(\tilde C)$ for $\tilde C \subset C$ a finitely generated field. 
\end{theorem}

\begin{proof}
 Let $m\colon X\rightarrow Y$ be a non-constant morphism of curves of degree $\geq 2$. Consider on $Y$ a 1-form $\eta$ with polar divisor $\sum _{i} m_i[y_i]$ (all $m_i>0$).
 Then $m^*\eta$ has polar divisor $\sum _{i}\sum _{x,\ m(x)=y_i}(m_i-1)e_x[x]$, where $e_x$ denotes the ramification index 
 of $m$ at $x\in X$. Using this formula, one concludes that $C$-vector space of the forms $\eta$ on $Y$ such that 
 $m^*\eta \in V(D)$ is bijectively mapped by $m^*$ to a proper $C$-linear subspace of $V(D)$. We want to show that
 $V(D)$ is larger than the union of all the proper $C$-linear subspaces, since the complement is the set of new forms in
 $V(D)$. \\
 
  There are only finitely many possibilities for $X\rightarrow \tilde{X}$ with the genus of $\tilde{X}$ greater than $1$. 
  This leaves us with considering morphisms to $\PP^1$ and to genus 1 curves. We first consider morphisms $m\colon X\rightarrow \mathbb{P}^1$. These are defined by a subfield $C(f)\subset C(X)$ for some $f\in C(X)\setminus C$. For convenience we identify $m$ and $f$ and moreover work modulo the action of the automorphism group of 
  $\mathbb{P}^1$. Now we consider the cases:\\
  (i). Suppose that the polar support of $\eta$ on $\mathbb{P}^1$ consists of a single point. Then $\eta$ is exact and so is $m^*\eta$ on $X$.
  Then $m^*\eta\in \{\dd h \ | \ h\in C(X), \ \div(df)\geq -D\}$ which is properly contained in $V(D)$.\\
  (ii). Suppose that $\eta$ has two simple poles. After changing variables and rescaling, we may assume $\eta =\frac{\dd x}{x}$. From $f^*\eta =\frac{\dd f}{f}\geq -D$ it follows that $f^{-1}(0)$ and $f^{-1}(\infty )$ are
  subsets of the support of $D$.

  The group $H:=\{f\in C(X)^*\ | \ {\rm support}(\div(f))\subset {\rm support}(D) \}/C^*$ is finitely generated, say by $f_1, \ldots, f_r$. Then for any $f \in H$, the 1-form $\dd f/f$ is a $\QQ$-linear combination of $\dd f_1/f_1, \ldots, \dd f_r/f_r$. Therefore, the set $\{c\frac{\dd f}{f} \ | \ c\in C,\ f\in H\}$ of old forms of this type is a ``small'' subset of 
 $V(D)$ since $\dim V(D)>1$ and $C\neq \mathbb{Q}$.\\ 
 (iii). Suppose that $\eta$ is not of the types discussed in (i) and (ii). Choose a coordinate on $\PP^1$ so that $0$ and $\infty$ are in the support of the polar divisor of $\eta$. The possible $f\in C(X)\setminus C$
 with $\div( f^*\eta)\geq -D$ have again the property that the support of $\div(f)$ lies in the support of the divisor
 $D$. Moreover the degree of the possible $f$ is bounded.  
 This shows that  there are only finitely many $f$ (up to scalars) such that $\div( f^*\eta )\geq -D$.  For each $f$ one obtains a proper $C$-linear subspace of $V(D)$ consisting of old forms.\\

In most cases, $C(X)$ contains only finitely many maximal proper subfields $F$ of genus $1$. In that case, they produce only finitely many proper subspaces of $V(D)$ of old forms.     
In special cases, $C(X)$ can contain infinitely many maximal proper subfields $F$ of genus 1. It can be seen that this set is countable  and so for an uncountable field $C$ the space $V(D)$ contains new forms.\\

In the case of a curve $X$ of genus $g\geq 2$ over a countable, algebraically closed field $C$, the set of the old forms in $V(D)$, coming from morphisms of $X$ to elliptic curves, is ``small'' in a certain sense. 

We sketch the idea for  the most extreme case, which is the case that
the Jacobian variety $\Jac(X)$ of $X$ is isogenous to  a product of $g$ copies of an elliptic curve $E$ which has complex multiplication. 
Morphisms from $X$ to an elliptic curve  $\tilde{E}$ factor over morphisms of $\Jac(X)$ to $\tilde{E}$. Up to isogenies, these morphisms
are homomorphisms $E^g\rightarrow E$ and have the form 
\[ (e_1,\dots ,e_g)\in E^g\mapsto A_1(e_1)+\cdots +A_g(e_g)\in E, \mbox{ where } A_1,\dots ,A_g\in \End(E).\]
 Now $\End(E)$ is a $\mathbb{Z}$-module of rank two. Using this one finds a subfield $\tilde{C}$ of $C$, finitely generated
over $\mathbb{Q}$, such that $X$, $D$, $E$, $\End(E)$ and all morphisms of $X$ to an elliptic curve are defined over $\tilde{C}$.
In particular, $V(D)=C\otimes_{\tilde{C}}V(D)_0$, where $V(D)_0$ is defined over $\tilde{C}$. For a morphism from $X$ to an 
elliptic curve, the corresponding $C$-linear subspace of old forms has the form $C\otimes_{\tilde{C}}W$ where 
$W$ is a proper $\tilde{C}$-linear subspace of $V(D)_0$. 
\end{proof}

\section{Formal solutions and initial values}\label{Section-Formal-Solutions}

 We consider $(X,\omega)$, a curve $X$ over $C$ with a rational 1-form $\omega \neq 0$ and the derivation $D$ dual 
 to $\omega$. Let $K$ denote the algebraic closure of $C\Lau{z}$ provided with the (continuous) derivation $\frac{\dd }{\dd z}$ and let $K^o$ denote the ring of integers of $K$.  {\it A formal solution} will be a solution of $(X,\omega)$ in $K$, i.e. a morphism $\phi\colon  \spec(K)\rightarrow X$ which is compatible with the derivations. This includes morphisms to closed points of $X$, i.e. constant solutions. These are of the form $\spec K \to \spec C \to X$, where the last map corresponds to a pole of $\omega$ in $X(C)$. We denote the set of formal solutions by $(X, \omega)(K)$. \\
 
 In this section, we show that formal solutions of $(X,\omega)$ correspond almost bijectively with points $X(C)$ of $X$. More precisely, we give a bijection $(X,\omega)(K)/\Gal(K/C\Lau{z}) \to X(C)$ that maps a formal solution to its initial value.\\ 

 Since $X$ is projective, irreducible and smooth, $X(K)=X(K^o)$ and any formal solution $\phi$ is induced by a local  $C$-linear homomorphism $O_{X,p}\rightarrow K^o$ compatible with the derivations, for a unique  $p\in X(C)$.
 The derivation $D$ produces a derivation $D\colon O_{X,p}\rightarrow C(X)$ which is continuous with respect to the topology induced by the valuation ring $O_{X,p}$. Therefore $D$ uniquely extends to a continuous derivation $D\colon C[[t]]\rightarrow C\Lau{t}$, where $C[[t]]$  denotes the completion of $O_{X,p}$. 
 
\begin{lemma}\label{4.1} Suppose that the continuous, $C$-linear, derivation $D$ satisfies $D(t)$ has valuation $r$ {\rm (}i.e., $r$ is the order of $D$). There is a parameter $T$ (i.e., $C[[t]]=C[[T]]$) such that one of the following holds:\\
{\rm (i)} $r=1$ and $D(T)=cT$ for some $c\in C^*$,\\
{\rm (ii)} $r < 1$ and $D(T)=T^r$,\\
{\rm (iii)} $r>1$ and $D(T)=T^r+cT^{2r-1}$ for some $c\in C$.
\end{lemma} 
 \begin{proof} Suppose that $r=1$. Suppose that $\frac{D(t)}{t}\equiv c\mod (t)$ with $c\in C^*$. 
 Write $T=te^{\sum _{n\geq 1}a_nt^n}$ with all $a_n\in C$.
  Then $\frac{D(T)}{T}=\frac{D(t)}{t}(1+\sum _{n\geq 1}a_nnt^n)$. Therefore the equation $\frac{D(T)}{T}=c$ leads to unique elements $a_n$. Further $T$ is unique up to multiplication by an element in $C^*$.\\
 
Suppose now $r<1$. A solution $T$ of $D(T)=T^r$ satisfies the equation $D(T^{-r+1})=-r+1$. Write $T^{-r+1}=\sum_{n\geq -r+1}a_nt^n$. The equation $-r+1=D(\sum _{n\geq -r+1}a_nt^n)=(\sum_{n\geq -r+1}na_nt^{n-1})D(t)$ determines clearly all $a_n$.
 Then $T=  (\sum _{n\geq -r+1}a_nt^n)^{1/(1-r)}$ is well defined and unique up to multiplication by $\zeta \in C$ with $\zeta^{-r+1}=1$. \\
 
 Suppose $r>1$. 
 Write $F=T^{r-1}$. To get $D(T) = T^r + cT^{2r-1}$ we have to solve $D(F)=(r-1)(F^2+cF^3)$ or equivalently $D(\frac{1}{F})=(1-r)(1+cF)$. Write 
 $\frac{1}{F}=\sum _{i\geq 1-r} a_it^i$ and $D(t)\sum _{i\geq r}b_it^i$. The equation now reads
 \[(\sum _{i\geq 1-r}ia_it^{i-1})\cdot (b_rt^r+b_{r+1}t^{r+1}+\cdots )=(1-r)(1+c(a_{1-r}^{-1}t^{r-1}+\cdots)).\]
 In the series of $D(\frac{1}{F})$ the term $t^{-1}$ is missing. This is compensated by the $c$ on the righthand side. 
  Thus  there is a unique solution. Then $T$ is a $(r-1)$th root of $F$.  \end{proof}

 Let $\phi$ be a formal solution. Then $\phi$ induces a homomorphism $O_{X,p}\rightarrow K^o$ for a unique point $p\in X(C)$, which in turn induces 
 $\hat {\phi}\colon \widehat{O_{X,p}}\rightarrow K^o$ compatible with the derivations. Let $r$ be the order of $D$ at $p$, i.e. the order of $D(t)$ for $t$ a local parameter at $p$. By \cref{4.1}, one can choose the parameter $T$ for $\widehat{O_{X,p}}$ such that:
 \begin{itemize}
     \item if $r > 1$, we have $D(T) = T^r + cT^{2r-1}$ for some $c \in C$. In this case, $\hat\phi(T) = 0$ and $\phi$ is a constant solution.
     \item if $r = 1$, we have $D(T) = cT$ for some $c \in C$. Again, $\hat \phi(T) = 0$ and $\phi$ is a constant solution.
     \item if $r = 0$, we have $D(T) = 1$ and so $\hat \phi(T) = z$. 
     \item if $r < 0$, we have $D(T) = T^r$, and $\hat\phi (T)=\zeta \cdot (1-r)^{\frac{1}{1-r}}z^{\frac{1}{1-r}}$ where $\zeta ^{1-r}=1$.
 \end{itemize}
 
 Put $e=\operatorname{lcm}(\{-r+1\})$, where $r$ ranges over  the pole orders of the derivation $D$. Let $(X, \omega)(K)$ and  $(X, \omega)(K^o)$ denote the 
 subsets of $X(K)$ and  $X(K^o)$ consisting of the formal solutions. The proposition below  easily follows from \cref{4.1} and the above remarks.
 
 \begin{proposition}\label{4.2} $(X, \omega)(K^o) = (X,\omega)(C[[z^{1/e}]])$. \\
 The canonical map $(X,\omega)(C[[z^{1/e}]])\rightarrow X(C)$, induced by
 $C[[z^{1/e}]]\rightarrow C$, produces a bijection $(X,\omega)(C[[z^{1/e}]])/\Gal \rightarrow X(C)$. Here $\Gal $ denotes the Galois group
 of $C\Lau{z^{1/e}}\supset C\Lau{z}$.
  \end{proposition}

\begin{remarks} 
 (1). One can state \ref{4.2} without reference to $e$ as:
{ \it The canonical map $(X,\omega)(K^o)\rightarrow X(C)$ induces a bijection 
$(X,\omega)(K^o)/\Gal \rightarrow X(C)$ where $\Gal $ is the Galois group of $K/C\Lau{z}$.}\\

\noindent (2). Let $p$ be a pole of $D$ of order $r$. Then there exists  a differential homomorphism
$\phi \colon O_{X,p}\rightarrow K$ which is not constant and has a pole (at $p$). Let $t\in O_{X,p}$ be a local parameter. Since
$\phi (t)\not \in K^o$, this homomorphism cannot be extended to the completion of $O_{X,p}$. 

 However, $\phi$ extends to $\phi \colon C(X)\rightarrow K$. Since $X$ is irreducible, smooth and complete there is a unique point
$q\in X(C)$ such that $\phi (O_{X,q})\subset K^o$. This leads to a differential homomorphism 
$\phi\colon \widehat{O_{X,q}}\rightarrow K^o$. We conclude that  $(X,\omega)(K)=(X,\omega)(K^o)$.

A simple example of this phenomenon is: \\
$X=\mathbb{P}_x^1$ and $D$ is given by $D(x)=x^5$.
Then $\phi \colon C(x)\rightarrow C\Lau{z^{1/4}}$ with $\phi (x)=(-4z)^{-1/4}$ is a differential homomorphism. Clearly $\phi$ does not extend to a homomorphism
$C[[x]]\rightarrow C\Lau{z^{1/4}}$. However, $\phi$ extends to a homomorphism $C[[\frac{1}{x}]]\rightarrow C\Lau{z^{1/4}}$ and the value $ (-4z)^{-1/4}$ for $\phi(x)$ leads to an element of $(\mathbb{P}^1, x^{-5}\dd x)(K^o)$.  
In fact, the set of ``all'' solutions of this equation is $\{ (c-4z)^{-1/4}| c\in C\}\cup \{0\}$. \\

 \noindent (3). $(X, \omega)(K^o)\rightarrow X(C)$ is bijective if and only if $D$ has no poles. Thus we have only the following cases:\\
 $g=0$ and $\omega =\dd x$ or $\omega =\frac{\dd x}{a x}$ with $a\in C^*$ or,\\
 $g=1$, $X$ has equation $y^2=x^3+ax+b$ and $\omega =c\frac{\dd x}{y}$.\\
 In the last case $(X, \omega)(K^o)$ is invariant under addition with $X(C)$. \\
 
 \end{remarks}

 \begin{remark} \label{4.5} {\it Can we describe ``all'' solutions of $(X,\omega)$?}  \\
 Let $\tilde C$ be an algebraically closed field properly containing $C$. By the above theory, points in $X(\tilde C) \setminus X(C)$ correspond to `additional' solutions in $\tilde C\Lau{z}\alg$. In the
 following example we discuss these extra solutions. 
 
 \begin{example}
 Consider the differential equation $u'= u^3 - 1$ corresponding to the pair $(\PP^1, \frac{\dd t}{t^3 - 1})$. Let $\tilde c \in \tilde C \setminus C$ be arbitrary, and let $y \in \tilde C\Lau{z}$ be the formal solution corresponding to the point $t = \tilde c$ in $\PP^1(\tilde C)$. 
 \end{example}

 
 Let again $K=C\Lau{z}\alg$. Consider the differential subalgebra $K[y]$ of $\tilde C\Lau{z}\alg$.
For any solution $\alpha \in K$, the maximal ideal $(y-\alpha)$ is invariant under $D = \frac{\dd}{\dd z}$ because $D(y-\alpha)=y^3-1 -(\alpha ^3-1)=y^3-\alpha ^3$. Thus there are infinitely many $D$-invariant maximal ideals.\\

 \noindent {\it Claim: $C$ is the field of constants of $K(y)$.}
 \begin{proof}
 A new constant has the form $f=a\prod (y-\alpha _i)^{n_i}$ with distinct elements $\alpha _i\in K$, $a\in K^*$, $n_i\in \mathbb{Z}$.
 Then $0=\frac{Df}{f}=\frac{a'}{a}+\sum _i n_i\cdot \frac{Dy-\alpha_i'}{y-\alpha _i}=
 \frac{a'}{a}+\sum _in_i\frac{y^3-1-\alpha_i'}{y-\alpha_i}$. Since
  $\alpha_i$ is not a pole of this expression one has $\alpha_i'=\alpha_i^3-1$. It follows that $\sum n_i=0$, $\sum n_i\alpha_i=0$
  and $\frac{a'}{a}+\sum n_i\alpha_i^2=0$. Since the $\alpha_i$ are known to be algebraically 
  independent over $C$,  we obtain a contradiction. \end{proof}

  From the fact that $K(y)\const = C$, it follows that there is a $K$-linear differential homomorphism from $K(y)$ to $\mathcal{U}$ for a suitable differentially closed field $\mathcal{U}$ with field of constants $C$. Indeed one can take for $\mathcal{U}$ the differential closure of the field $K(y)$. In this way, we can produce many additional solutions for this differential equation, without increasing the subfield of constants.
  We expect the same phenomenon to occur for all autonomous equations of general type.

  This is in contrast with the situation for systems of linear equations and for equations of exact, exponential and Weierstrass type: for those types of equations, adding new solutions to $K$ introduces new constants. 
  \end{remark}

\section{\texorpdfstring{$D^n$}{Dn}-finiteness of solutions}

 Let $K$ be a differential field (of characteristic 0) with algebraically closed field of constants $C$.
 An extension of differential fields $F\supset K$ is called an {\it iterated Picard--Vessiot extension}
 if there exists a sequence of differential fields  $F=K_n\supset \cdots \supset K_1\supset K_0=K$ where each extension $K_{i+1}\supset K_{i}$ is a Picard--Vessiot extension.

 From the literature, especially \cite{Jimenez-Pastor-Pillwein}, we adopt the following definitions:\\
 An element $\xi \in C\Lau{z}$ is {\it $D$-finite} (over $C(z)$) if it satisfies a linear scalar differential equation with coefficients
 in $C(z)$. For $n > 1$ we inductively define $\xi$ to be {\it $D^n$-finite} (over $C(z)$) if $\xi$ satisfies
 a linear scalar differential equation with coefficients which are $D^{n-1}$-finite.
 An element $\xi\in C\Lau{z}$ is called {\it $D$-algebraic} (over $C(z)$) if there is a non-trivial polynomial relation
 $P(\xi ^{(m)},\xi^{(m-1)},\dots ,\xi,z)=0$ with coefficients in $C$ between the derivatives of $\xi$ and $z$.\\

 One easily verifies that a $D^n$-finite element $\xi$ lies in the iterated Picard--Vessiot extension of $C(z)$ defined
 by all the linear scalar equations involved in the $D^n$-property of $\xi$.\\

 In \cite{Jimenez-Pastor-Pillwein} it is shown that $D^n$-finiteness defines strictly increasing subsets of $C\Lau{z}$. A natural question is: 
 
 {\it Is there a $D$-algebraic element which is not $D^n$-finite for any $n\geq 1$}?\\
 The positive answer is given in \cref{dfinite}. \\
 
  \begin{proposition}  \label{dfinite}
 Let $\xi\in C\Lau{z}$ be a non-constant solution of an autonomous first order equation of general type.
  Then $\xi$ is not contained in any iterated Picard-Vessiot extension of $C(z)$. 
  In particular, $\xi$ is not $D^n$-finite over $C(z)$ for any $n$. 
  \end{proposition}
 \begin{proof}  Assume that $\xi\in C\Lau{z}$ is a non-constant solution of the autonomous equation of general type $(X,\omega)$. Suppose that $\xi$  lies in the iterated Picard--Vessiot extension 
 $F=K_n\supset \cdots \supset K_1\supset K_0=C(z)$. We will deduce a contradiction by induction on $n$.\\
 
 Let $G$ be the differential Galois group of $K_n\supset K_{n-1}$. For $\sigma \in G$, the element $\sigma(\xi)$
 is a solution of $(X,\omega)$. Since $K_n$ is of finite transcendence degree over $C$, \cref{finitely many solutions} shows that there are only finitely many
 possibilities for $\sigma (\xi)$. 
 
  We conclude that $\xi$ lies in the fixed field $K_n^{G^o}$ where $G^o$ denotes the 
  connected component of the identity of $G$. Thus $\xi$ is algebraic over $K_{n-1}$. For $n=1$, this yields the contradiction that
 $\xi$ is algebraic over $C(z)$.\\
 
   For $n>1$ we consider $P:=T^m+a_{m-1}T^{m-1}+\cdots +a_0$, the minimum polynomial  of $\xi$ over
   $K_{n-1}$. Let $\sigma$ be an element of the differential Galois group of $K_{n-1}\supset K_{n-2}$. The element
   $\sigma$ extends to a differential automorphism $\tilde{\sigma}$ of the algebraic closure $K^{alg}_{n-1}$ of $K_{n-1}$.
   Then $\tilde{\sigma}(\xi)\in K_{n-1}^{alg}$ is a solution of $(X,\omega)$ and is a zero of 
   $\sigma(P)=T^m+\sigma(a_{m-1})T^{m-1}+\cdots +\sigma(a_0)$. 
   There are only finitely many possibilities for $\tilde{\sigma}(\xi)$. Then there are  also only finitely many possibilities
   for $\sigma(P)$ and so $a_{m-1},\dots ,a_0$ are algebraic over $K_{n-2}$.  Hence $\xi$ is algebraic over $K_{n-2}$.
   The normalization $L$ of the algebraic extension $K_{n-2}(\xi)\supset K_{n-2}$ is a Picard--Vessiot extension.
    Now
   $\xi\in L\supset K_{n-2}\supset \cdots \supset K_0$. By induction this yields a contradiction.   \end{proof}
 
 An explicit example of a function that is $D$-algebraic but not $D^n$-finite for any $n$, is any non-constant solution $\xi \in C\Lau{z}$ of the equation
$u'=u^3-u^2$. Clearly $\xi$ is $D$-algebraic. The key property used in the proof of \cref{dfinite} is algebraic independence between non-constant solutions of a new and general type differential equation. In the particular case of the equation $u'= u^3 - u^2$, there is a direct proof of this property, first given by Rosenlicht (\cite[Corollary to Proposition 1]{R1}) and relying on a result of James Ax (\cite[Proposition 2]{Ax}. We give a self-contained version of this proof below.

\begin{lemma}\label{Lemma Rosenlicht}
Let $F$ be a differential field with field of constants $C$.
If non-constant solutions $x_1,x_2\in F$ of the equation $u'=u^3-u^2$ are algebraically dependent over $C$, then
$x_1=x_2$.
\end{lemma}

\begin{proof} Suppose that $x_1$ and $x_2$ are algebraically dependent over $C$.  Then $K:=C(x_1,x_2)\subset F$ is a differential subfield of transcendence degree 1 over its field of constants $C$. Let $Y$ be the irreducible, projective curve over $C$ with function field $K$. 

For $i=1,2$ we consider the 1-form $\omega _i:=\frac{dx_i}{x_i^3-x_i^2}=\frac{d (\frac{x_i-1}{x_i})}{(\frac{x_i-1}{x_i})}+dx_i^{-1}$ on 
$Y$. Let $h\in K^*$ be such that $\omega_1=h\cdot \omega_2$. 
Let $D$ be the derivation on $K$. 
Then $D$  acts in a natural way on 1-forms, namely 
$D(f\cdot dg)=D(f)\cdot dg +f\cdot d(Dg)$. One easily computes 
$D(\omega _i)=0$ for $i=1,2$. Then also $D(h)=0$ and so $h\in C$.  

The 1-form  $\tau:=\frac{d( \frac{x_1-1}{x_1})}{(\frac{x_1-1}{x_1})}-h \cdot \frac{d( \frac{x_2-1}{x_2})}{(\frac{x_2-1}{x_2})}=
h\cdot dx_2^{-1}-dx_1^{-1}$ on $Y$ has residues zero because it is exact. The residues of the 1-form 
$\frac{d( \frac{x_i-1}{x_i})}{(\frac{x_i-1}{x_i})}$ are non-zero integers. It follows that $h\in \mathbb{Q}$. 

Write $h=\frac{a}{b}$ with $\gcd (a,b)=1$ and put $f:=(\frac{x_1-1}{x_1})^a(\frac{x_2-1}{x_2})^{-b}$. Then $\frac{df}{f}=d(ax_2^{-1}-bx_1^{-1})$. Now $f$ has no poles or zeros and thus $f\in C^*$. Then
$ax_2^{-1}-bx_1^{-1}\in C$ and $C(x_1)=C(x_2)$. Finally,  $x_1=x_2$ follows from the observation that the only 
$C$-linear automorphism of $C(x)$ which leaves $\frac{dx}{x^3-x^2}$ invariant is the identity.
\end{proof}

\begin{proposition}\label{Rosenlichts example} Let $F$ be a differential field with field of constants $C$. If $x_1,\dots ,x_n \in F$ are distinct non-constant solutions of 
$u'=u^3-u^2$ in $F$, then $x_1,\dots ,x_n$ are algebraically independent over $C$.
\end{proposition}
 
\begin{proof} Induction on $n$. \cref{Lemma Rosenlicht} is the case $n=2$. The proof of the induction step 
is obtained by copying the proof of the Lemma with $C$ replaced by the differential field $L:=C(x_3,\dots ,x_n)$. 
Let $\tilde{C}$ denote the field of constants of $L(x_2)$. If $\tilde{C}$ is larger than $C$, then $x_2,\dots ,x_n$ are
algebraically dependent over $\tilde{C}$. This is ruled out by the induction hypothesis and so $C$ is the field of
 constants of $L(x_2)$.  Assume that $x_1,x_2$ are algebraically dependent over $L$, then $x_1$ is algebraic over
 $L(x_2)$ and so $C$ is the field of constants of $L(x_1,x_2)$. As in the proof of the Lemma this implies
 $x_1=x_2$.  
 \end{proof}

 \section*{Acknowledgements}
 The results of this manuscript were inspired by lectures at the Workshop on Differential Galois Theory and Differential Algebraic Groups at the Fields Institute in Toronto in 2017. We thank the organisers for this inspiring event. We also thank Bas Edixhoven and Pierre Parent for their interest in our work.
 
 We are grateful to Antonio Jim\'enez-Pastor and Veronika Pillwein for posing a question to one of us, which led to Section~7 of this text.

\bibliographystyle{plain}

\end{document}